\documentclass[reqno,12pt]{amsart}

\usepackage{amsmath}
\usepackage{amssymb}
\usepackage{amsthm}
\usepackage{graphicx}
\usepackage{psfrag}
\usepackage{verbatim}
\newcommand{\bbE}{\mathbb{E}}

\newcommand{\bbX}{\mathbb{X}}

\newcommand{\bbP}{\mathbb{P}}

\newcommand{\bbR}{\mathbb{R}}

\newcommand{\eps}{\epsilon}

\newcommand{\ol}{\bar}

\newcommand{\wh}{\widehat}

\newcommand{\cA}{\mathcal A}

\newcommand{\cZ}{\mathcal Z}

\newcommand{\rmc}{{\rm c}}
\newcommand{\rmC}{{\rm C}}
\newcommand{\rmF}{{\rm F}}

\newcommand{\dd}{\,{\rm d}}
\newcommand{\supp}{\rm supp}
\newcommand{\sgn}{\rm sgn}
\newcommand{\wt}{\widetilde}

\newtheorem{theorem}{Theorem}

\newtheorem{lem}[theorem]{Lemma}
\newtheorem{prop}[theorem]{Proposition}

\title[LD for The Empirical Distribution in The BRW]{Large Deviations for The Empirical Distribution in The Branching Random Walk}
\author{Oren Louidor}
\address{Mathematics Department\\
 University of California Los Angeles\\
	Los Angeles, CA}
\email{louidor@math.uncla.edu}
\thanks{The research of the first author was supported by a Simons Postdoctoral Fellowship}

\author{Will Perkins}
\address{School of Mathematics\\
 Georgia Institute of Technology \\
	Atlanta, GA}
\email{perkins@math.gatech.edu}
\thanks{The research of the second author was supported in part by NSF grant OISE-0730136 and an NSF Postdoctoral Fellowship}

\begin{document}

\maketitle

\begin{abstract}
We consider the branching random walk $(Z_n)_{n \geq 0}$ on $\bbR$ where the underlying motion is of a simple random walk and branching is at least binary and at most decaying exponentially in law. It is well known that $\ol{Z_n}(A) \to \nu(A)$ almost surely as $n \to \infty$ for typical $A$'s, where $\ol{Z_n}$ is the empirical particles distribution at generation $n$ and $\nu$ is the standard Gaussian measure on $\bbR$. We therefore analyze the rate at which $\bbP(\ol{Z}_n(A) > \nu(A) + \epsilon)$ and $\bbP(\ol{Z}_n(A) < \nu(A) - \epsilon)$ go to zero for any $\epsilon>0$. We show that the decay is doubly exponential in either $n$ or $\sqrt{n}$, depending on $A$ and $\epsilon$ and find the leading coefficient in the top exponent. To the best of our knowledge, this is the first time such large deviation probabilities are treated in this model.
\end{abstract}

\section{Introduction and Results}
In this work we analyze the decay of probabilities of certain unlikely deviation events involving the Branching Random Walk (henceforth BRW). As far as we know, very little has been done in this direction, although, after optimal law of large numbers and central limit theorem type results have been fully obtained, both the question and the events we consider seem to us natural and fundamental. To fix notation and context, we begin by briefly describing the model (\ref{sub:Setup}) and giving a short account of some of the relevant results in its analysis (\ref{sub:OldResults}). A precise statement of the contribution in this paper then follows
(\ref{sub:NewResults}), and finally the idea in the proof of the main theorem is conveyed (\ref{sub:IdeaOfProof}). Complete proofs for all statements are given in Section~\ref{sec:Proofs}.

\subsection{Setup}
\label{sub:Setup}
The BRW model traces the evolution by means of reproduction and motion of a population of particles on the real line, carried out synchronously in discrete steps or {\em generations}. We denote by $Z_n$ (henceforth the {\em particles measure}) the population at time $n = 0,1, \dots$, which we describe as a point measure on $\bbR$ with a mass $1$ per particle. The process is formally defined as follows. Initially there is a single particle at the origin $Z_0 = \delta_0$. It evolves in one generation to a random point measure $Z_1$. Although one may consider any law for $Z_1$, often and in this paper as well, attention is restricted to evolution by means of independent reproduction and motion. That is, $Z_1$ is realized by the particle giving birth to a random number of descendants, dying, and then all descendants independently of each other and of their number moving according to some common spatial distribution $F$.

At any further generation $n \geq 2$ we have (conditioned on $Z_{n-1}$),
\begin{equation}
  Z_n = \sum_{x \in Z_{n-1}} \wt{Z}^x_1 \,,
\end{equation}
where $\wt{Z}^x_1(\cdot)$ has the same distribution as $Z_1(\cdot - x)$ and
$\{\wt{Z}^x_1 :\: x \in Z_{n-1}\}$ are independent. Here and later, for a point measure $\zeta$ with integer masses, we write $x \in \zeta$ iff $x$ is an atom of $\zeta$, that is if $\zeta(x) := \zeta(\{x\}) > 0$. We use $(x :\: x \in \zeta)$ for the multi-set of atoms of $\zeta$, where each atom $x$ is repeated $\zeta(x)$ times. Moreover, if this multi-set is used as an index set (as above), different copies of the same atom are considered different indices. 

Despite the old age of this model it is still quite central in pure and applied probability. It remains a popular model for describing and analyzing phenomena in various applied disciplines, such as biology, population dynamics and computer science. At the same time, due to the fundamentality of the stochastic dynamics it captures, it 
is frequently found in various seemingly unrelated mathematical models (e.g. the Gaussian Free Field~\cite{ZB2012}, Interacting Particle System~\cite{Liggett}). Finally, there are aspects of the model which are still not understood or only beginning to be understood now (e.g. its extremal process~\cite{Bovier}). For the classical theory of BRW, we direct the reader to the survey by Ney~\cite{ney} and the books by R\'{e}v\'{e}sz~\cite{revesz} and Harris~\cite{Harris}.

\subsection{Known Results}
\label{sub:OldResults}
Since the population-size process $(|Z_n|)_{n \geq 0} = (Z_n(\bbR))_{n \geq 0}$ is a standard Galton Watson process,
it is well known that once reproduction is super-critical
\begin{equation}
\label{eqn:supercritical}
  \beta := \bbE |Z_1| > 1 
\end{equation}
and assuming 
\begin{equation}
\label{eqn:KS}
  \bbE |Z_1| \log |Z_1| < \infty
\end{equation}
then for the {\em normalized particles measure} $\wh{Z}_n = \beta^{-n} Z_n$ we have
almost surely
\begin{equation}
\label{e:1003}
  \lim_{n \to \infty} |\wh{Z}_n| = |\wh{Z}| \,,
\end{equation}
where $|\wh{Z}|$ is some non-negative random variable with $\bbE|\wh{Z}| = 1$.
The optimal version of this theorem is due to Kesten and Stigum~\cite{KestenStigum66}. If $\beta \leq 1$, the population dies out with probability $1$; hence from now on, we shall assume
\eqref{eqn:supercritical}.

When displacement is considered as well, an analogous result to the above, conjectured by Harris~\cite{Harris}, first proved by Stan~\cite{Stam66}, and then proved under optimal conditions by Kaplan~\cite{Kaplan82} is 
\begin{equation}
\label{e:1004}
  \lim_{n \to \infty} \wh{Z}_n(\sqrt{n}A) = |\wh{Z}| \nu(A)
    \qquad \bbP \text{-a.s.} \,,
\end{equation}
Here $A \in \cA_0$ where
\begin{equation}
\label{eqn:10045}
  \cA_0 := \{(-\infty, x] :\: x \in \bbR\} \, ,
\end{equation}
$\nu$ is the standard Gaussian measure on $\bbR$, and the assumptions are \eqref{eqn:supercritical}, \eqref{eqn:KS} for branching, and zero mean and unit variance for the motion, that is
\begin{equation}
  \int x \dd \rmF(x) = 0	\quad ; \qquad \int x^2 \dd \rmF(x) = 1 \, .
\end{equation}
Combining \eqref{e:1003} and \eqref{e:1004} and denoting the {\em empirical particles distribution} by $\ol{Z}_n = Z_n / |Z_n|$, we have
\begin{equation}
\label{eqn:1007}
  \lim_{n \to \infty} \ol{Z}_n(\sqrt{n}A) = \nu(A) \,.
\end{equation}

Once leading order asymptotics \eqref{e:1003}, \eqref{e:1004} have been obtained, second-order terms, or the question of the rate of the convergence, can be approached. For the population size, Heyde~\cite{Heyde71} has shown that under $\bbE |Z_1|^2 < \infty$, for some (explicit) $\alpha_0 > 0$, as $n \to \infty$
\begin{equation}
  \alpha_0 |\wh{Z}|^{-1/2}\beta^{n/2}(|\wh{Z}_n| - |\wh{Z}|) \Rightarrow N(0, 1) \,.
\end{equation}
For the particles measures, more recently Chen~\cite{chen2001exact} has proved that for all $A \in \cA_0$,
\begin{equation}
  \sqrt{n} \big(\wh{Z}_n(\sqrt{n}A) - |\wh{Z}|\nu(A) \big) = \varphi_1(n)|\wh{Z}| + \alpha_1 \wh{M} + o(1) \,,
\end{equation}
as $n \to \infty$, where $\alpha_1 > 0$, $\varphi_1(\cdot)$ is a bounded function, and $\wh{M}$ is some random variable - all explicitly defined. In the case he considered, motion is of a simple random walk and branching admits the same assumptions as in Heyde's.

Having settled the main questions in the ``typical deviations'' regime, it is natural to turn to the regime of atypical or large deviations. Results here are not as abundant. For $|Z_n|$, Athreya~\cite{ATH1} has considered the following probabilities:
\begin{equation}
  \bbP \big(\big| |Z_{n+1}|/|Z_n| - \beta \big| > \Delta \big) \quad \text{and} \qquad 
  \bbP \big(\big| |\wh{Z}_n| - |\wh{Z}| \big| > \Delta \big) \,,
\end{equation}
for $\Delta > 0$ and under the assumptions of exponential moments and $|Z_1| \geq 1$. If 
$p := \bbP(|Z_1| = 1) > 0$, he showed that the probability on the left is 
\begin{equation}
\label{eqn:10075}
\lambda_0(\Delta) p^n (1+o(1)) 
\end{equation}
for some explicitly defined $\lambda_0(\Delta)> 0$ and otherwise, it is at most 
\begin{equation}
\alpha_1(\Delta) \exp (-\lambda_1(\Delta) b^n) \,,
\end{equation}
where $b$ is the first integer for which $\bbP(|Z_1| = b) > 0$ and $\lambda_1(\Delta), \alpha_1(\Delta) > 0$. For the probability on the right, he obtained the bound
\begin{equation}
  \bbP \big(\big| |\wh{Z}_n| - |\wh{Z}| \big| > \Delta \big) \leq C \exp \big( 	-C^\prime \Delta^{2/3} (\beta^{1/3})^n \big) \,.
\end{equation}
Above $C, C^\prime > 0$ are some universal constants. See also~\cite{NeyVidyashankar2004}. Different atypicality is treated by Jones~\cite{Jones04} and Biggins and Bingham~\cite{BigginsBingham93} who investigate the left and right tail of $|\wh{Z}|$.

For the BRW, much effort has been directed into estimating the number of particles which deviate linearly away from the mean displacement in the underlying motion. It is a classical result by Biggins~\cite{biggins1977chernoff} that for any $A \in \cA_0$,
\begin{equation}
  \lim_{n \to \infty} n^{-1} \log Z_n(nA) = -\inf_{x \in A} \Lambda^*(x)
    \qquad \bbP \text{-a.s.} \,,
\end{equation}
if the r.h.s. is positive and otherwise $Z_n(nA) \to 0$ a.s. Here
$\Lambda^*$ is the Legendre-Fenchel transform of $\Lambda(\theta) = \log \bbE \int e^{\theta x} \dd Z_1(x)$, which is assumed to be finite. This can be also used to obtain the speed of the left (or right) most particle as $\inf \{x: \: \Lambda^*(x) < 0\}$, although to obtain sharper results, different methods have been used (c.f. Brahmson~\cite{Bramson78BM, Bramson83}, and Addario-Berry and Reed~\cite{AddarioBerryReed09}). 

Perhaps closest to the type of large-deviation analysis we do here is the result by Athreya and Kang in~\cite{AthreyaKang98}, where instead of a motion in $\bbR$, particles move according to some positive-recurrent Markov chain with invariant measure $\pi$. Along with a local version of \eqref{eqn:1007}, they find that the probability that at time $n$ the fraction of particles at state $s$ is at least $\Delta > 0$ away from $\pi(s)$ decays exponentially as $\lambda(\Delta) p^n$ for some explicit $\lambda(\Delta) >0 $ and with $p$ as in \eqref{eqn:10075}, which is assumed
to be positive Nevertheless, this is still quite far from what we do here. First, random walk is typically null recurrent (unless degenerate). Second, there is no spatial component (e.g. CLT-type phenomenon) to their problem. Third, we in fact assume $p_1 = 0$ and thus obtain very different decay scales.

\subsection{New Results}
\label{sub:NewResults}
In this work we analyze large deviation probabilities of the form:
\begin{equation}
\label{eqn:1008}
  \bbP(|\ol{Z}_n(\sqrt{n}A) - \nu(A)| > \Delta) \,.
\end{equation}
for some $\Delta > 0$. In light of \eqref{eqn:1007}, the above clearly decays in $n$ and we aim to understand how fast. 

\subsubsection*{Assumptions.}  We make the following assumptions. For branching, we shall assume that $|Z_1|$ is non-deterministic, that $\bbE e^{\theta |Z_1|} < \infty$ for $\theta$ in some neighborhood of $0$ and that $\bbP(|Z_1| \geq 2) = 1$. The last condition guarantees that exponential growth of the population size is unavoidable. Although the case of $\bbP(|Z_1| \geq 2) < 1$ is an interesting problem, it is of a different nature as it permits using strategies which suppress the branching in order to realize large deviation events. This will result in a different scale for the decay in \eqref{eqn:1008}. For the underlying motion, we shall assume simple random walk steps. The precise step distribution will not change the result, as long as it has mean zero and bounded or sufficiently decaying tails. Again, allowing for steps with fat tails would have given rise to strategies which exploit these tails for achieving the unlikely events, resulting again in a problem of a different nature and a different scale for the decay of \eqref{eqn:1008}.

We are now ready to state our main result. Let $\cA$ be the algebra generated by $\cA_0$ (defined in \eqref{eqn:10045}). For $A \in \cA$ non-empty and $p \in (0,1)$ define
\begin{eqnarray}
 \tilde{I}_A(p) & = &  \inf \{ |x| :\: \nu( A-x) \geq p \,,\,\, x \in \bbR\} \, , \\
 \tilde{J}_A(p) & = &
    \inf \big\{r: \:  \sup_{x \in \bbR}  \nu \big((A-x)/\sqrt{1-r} \big) \geq p \,,\,\, r \in [0,1) \big\} 
\end{eqnarray}
and with $b=\min\{k : \: \bbP(|Z_1| = k) > 0\} \geq 2$, set
\begin{eqnarray}
  I_A(p) & = 	& (\log b)\tilde{I}_A(p) \,,\, \quad \qquad \qquad \qquad \qquad \qquad \qquad \qquad \\
  J_A(p) & = 	& (\log b) \tilde{J}_A(p) \, .
\end{eqnarray}
Then,
\begin{theorem}
\label{thm:GeneralCase}
For $A \in \cA$ non-empty and $p \in (0,1)$ such that $p > \nu(A)$,
\begin{equation}
\label{e:1}
  \log \big[-\log \bbP \big(\ol{Z}_n(\sqrt{n}A) \geq p \big)\big] \sim \left \{
    \begin{array}{lcl}
      I_A(p)\,\sqrt{n} 	& \quad & \text{if } I_A(p) < \infty	\\
      J_A(p)\,n 	& 	& \text{otherwise.}
    \end{array}
    \right.
\end{equation}
as $n \to \infty$.
\end{theorem}

Replacing $A$ with $A^c$ in Theorem~\ref{thm:GeneralCase}, 
one has 

\medskip
\noindent {\bf Theorem 1$^\prime$}.
For all $A \in \cA \setminus \bbR$ and $p \in (0,1)$ such that
$p < \nu(A)$,
\begin{equation}
\label{e:1a}
  \log \big[-\log \bbP \big(\ol{Z}_n(\sqrt{n}A) \leq p \big)\big] \sim \left \{
    \begin{array}{lcl}
      I_{A^c}(1-p)\,\sqrt{n}	& \quad & \text{if } I_{A^c}(1-p) < \infty	\\
      J_{A^c}(1-p)\, n		& 	& \text{otherwise.}
    \end{array}
    \right.
\end{equation}
as $n \to \infty$.

\medskip 
As follows from Proposition~\ref{prop:supx} below, for $A$ and $p$ as in 
the conditions of the theorems either $I_A(p) \in (0,\infty)$ or $I_A(p) = \infty$ 
and $J_A(p) \in (0,\log b)$. Thus on a double-exponential scale, Theorem~\ref{thm:GeneralCase} and 1$^\prime$ capture the right first-order asymptotics 
for the decay of the probability of a large deviation in the empirical distribution
for such $A$'s and $p$'s.
 
The statement in the theorem still holds if we replace
the weak inequality in \eqref{e:1} or \eqref{e:1a} by a strong one. Our proof for the 
lower bound on $\bbP(\ol{Z}_n(\sqrt{n}A) \geq p)$ essentially works for 
$\bbP(\ol{Z}_n(\sqrt{n}A) > p)$.

The restriction to intervals of the form $(-\infty ,x]$, $(y, x]$ and
$(y, \infty)$ in $\cA$ is quite arbitrary and the theorem still holds if $\cA$
is the algebra generated by sets of the form $(\infty, x)$ or more generally,
the set of all finite unions of disjoint intervals which either contain their endpoints or do not, or contain only one of them and can be finite or infinite,
but as long as their interior is non-empty. 

On the other hand, \eqref{e:1} cannot be expected to hold for all Borel sets, nor even all continuity sets of $\nu$. Indeed, the following shows 
that there are simple enough sets for which the decay in \eqref{eqn:1008} has neither linear nor radical rate on a double exponential scale.
\begin{prop}
\label{prop:Interpolation}
For all $\alpha \in (1/2,1)$ and $p \in (0,1)$, there exists a set $A$, which is a countable union of disjoint finite intervals, such that
\begin{equation}
\label{eqn:111}
 \log \big[-\log \bbP \big(\ol{Z}_n(\sqrt{n}A) \geq p \big)\big] \sim  n^{\alpha}  
\end{equation}
\end{prop}

Similarly, the restriction in our main theorem to values of $p$ in $(0,1)$ is essential.
In Theorem~\ref{thm:GeneralCase}, for instance, in the case $p=0$ the probability in the l.h.s. of \eqref{e:1} does not decay, and for certain sets in $\cA$, the case $p=1$ cannot be handled by the current proof nor a straightforward modification of it.

\subsection{Idea of Proof}
\label{sub:IdeaOfProof}
It is usually the case in the realm of large deviations that obtaining decay asymptotics for probabilities of unlikely events amounts to finding (and proving that it is such) an optimal (that is least ``costly'' in terms of probability) ``strategy'' for realizing the unlikely event. Consider therefore $A \in \cA$ and $p \in (\nu(A),1)$ as in the conditions of Theorem~\ref{thm:GeneralCase}. What is the optimal strategy for having at least $p$ fraction of the population in the set $\sqrt{n}A$ at time $n$ instead of the likely $\nu(A)$?

As it turns out, among all possible strategies one needs to consider only two: a {\em shift} strategy and a {\em dilation} strategy. In the former, all particles move together in either the left or right direction for $w=|x|\sqrt{n}$ generations (up to integer rounding, $x \in \bbR$). This can be done with probability 
$\exp(-b^{|x|\sqrt{n}(1+o(1)})$ by keeping the number of particles at its minimum. Relative to the position of the particles at generation $w$, the target set has now ``shifted'' by $-x\sqrt{n}$. Therefore after dividing by the CLT scaling of $\sqrt{n}$, each particle at generation $w$ will typically have (asymptotically) a fraction of $\nu(A-x)$ of its descendants in $\sqrt{n}A$, and this will also be the fraction for the entire population. Consequently, if there exists $x$ for which $\nu(A-x) \geq p$, this strategy will realize the event $\{\ol{Z}_n(\sqrt{n}A) \geq p\}$ at the sole cost of ``steering'' the population for $w$ generations. This cost is $\exp (-e^{I_A(p)\sqrt{n}(1+o(1)})$ once $x$ is chosen closest to $0$.

If there is no $x$ for which $\nu(A-x) \geq p$, a ``dilation'' strategy is employed, whereby all particles move together for $w^\prime = r^\prime n +x^\prime\sqrt{n}$
generations ($x^\prime \in \bbR$, $r^\prime \in (0,1)$) such that at generation $w^\prime$ they are all at position $x^\prime \sqrt{n}$ . If $r^\prime, x^\prime$ are chosen such $\nu((A-x^\prime )/\sqrt{1-r^\prime }) \geq p$ then as in the shift case, the typical overall fraction in $\sqrt{n}A$ at a large time $n$ will be at least $p$. The probabilistic cost of this strategy is therefore incurred just in the first $w^\prime$ generations, and by keeping reproduction at its minimum, it can be $\exp(-b^{r^\prime n (1+o(1))})$. Choosing the smallest $r^\prime $ possible, $\{\ol{Z}_n(\sqrt{n}A) \geq p\}$ can be achieved by a strategy which has probability $\exp(-e^{-J_A(p)n(1+o(1)})$.

Of course these strategies only give lower bounds for the probability in question. One therefore must also show that other strategies would not cost less. In addition, to make the above heuristics precise, our proof requires certain uniform estimates for the probabilities of finding typical fractions as well as coarse (a priori) estimates for finding atypical ones.

\section{Proofs}
\label{sec:Proofs}
In this section we provide proofs for the statements in (\ref{sub:NewResults}). We first introduce further notation (\ref{sub:Notation}) which will be used in the proofs then prove various preliminary statements (\ref{sub:Perliminiaries}) which are required in order to make the ideas from (\ref{sub:IdeaOfProof}) precise. We then prove the main theorem (\ref{sub:ProofOfMainTheorem}) and finally prove Proposition~\ref{prop:Interpolation} (\ref{sub:countersection}). 

\subsection{A bit more notation}
\label{sub:Notation}
The space of all particles measures, that is, finite point measures on $\bbR$ with integer masses, will be denoted by $\cZ$. For $\zeta \in \cZ$, we denote by
$(Z_n^\zeta)_{n \geq 0}$ a BRW process with a similar evolution as $(Z_n)_{n \geq 0}$, only that initially $Z_0 = \zeta$. We will write $Z_n^x$ in place of $Z_n^{\delta_x}$ for short. $\nu_n$ is the distribution of the position of a simple random walk after $n$ steps. For $u \in \bbR$, as usual, $u^+ = \max(0, u)$ and $u^- = -(-u)^+$. We will use $C$, $C^\prime$, $C^{\prime\prime}$ to denote positive constants whose value is immaterial and changes from one use to the other. Constant values which are used more than once are denoted $C_0, C_1,..$, and their values become fixed the first time they appear in the text.

\subsection{Preliminaries}
\label{sub:Perliminiaries}
\begin{prop}
\label{prop:supx}
Let $A \in \cA$ be non-empty and $p \in (0,1)$.
\begin{enumerate}
  \item \label{prop:supx0}
    $(\rho, \xi) \mapsto \nu(\rho A + \xi) \in \rmC^{\infty}(\bbR^2)$.
  \item \label{prop:supx1}
    If $\tilde{I}_A(p) \in [0, \infty)$ then there exists $x \in \bbR$ with $|x| = \tilde{I}_A(p)$ such that
    \begin{equation}
      \nu(A-x) \geq p \, .
    \end{equation}
  \item \label{prop:supx2}
    $\tilde{J}_A(p) \in [0,1)$ and there exists $x \in \bbR$ such that with $r=\tilde{J}_A(p)$
    \begin{equation}
      \label{e:3}
      \nu \big((A-x)/\sqrt{1-r} \big) \geq p \, .
    \end{equation}
  \item \label{prop:supx3}
    If $p > \nu(A)$ then either $\tilde{I}_A(p) \in (0,\infty)$ or
    $\tilde{I}_A(p) = \infty$, $\tilde{J}_A(p) \in (0,1)$.
\end{enumerate}
\end{prop}

\begin{proof}
Part~\ref{prop:supx0} follows from the dominated convergence theorem and standard arguments once we write 
\begin{equation}
  \nu(\rho A + \xi) = \int_A \frac{1}{\sqrt{2\pi}} e^{-\frac{(\rho t+\xi)^2}{2}} \rho \dd t
\end{equation}
since the integrand is in ${\rm C}^{\infty}(\bbR^2)$.

For part~\ref{prop:supx1} and~\ref{prop:supx2}, if $A=\bbR$, then $\tilde{I}_A(p) = \tilde{J}_A(p) = 0$, and there is nothing to prove. Otherwise, define
\begin{equation}
  \varphi_A(r,x) = \nu((A-x)/\sqrt{1-r})
\end{equation}
which is in $\rmC^\infty([0,1) \times \bbR)$ by part~\ref{prop:supx0}.
Therefore $\{x \in \bbR :\: \varphi_A(0, x) \geq p\}$ is a closed set, which, if non-empty, must contain a minimizer of $|\cdot|$. This shows part~\ref{prop:supx1}.

For part~\ref{prop:supx2}, if $A$ contains a half-infinite interval, then since $\varphi_A(0, x) \to 1 > p$ if $x \to +\infty$ or $x \to -\infty$, we must have $\tilde{I}_A(p) < \infty$. Therefore $\tilde{J}_A(p) = 0$, and \eqref{e:3} is satisfied 
with $r=0$ and $x$ from part~\ref{prop:supx1}. Otherwise, $A$ is a finite union of finite intervals, and so there must exist $R < 1$, $M < \infty$ such that
\begin{itemize}
 \item $\varphi_A(r,x) \geq p$ for some $0 \leq r \leq R$ and $x$ with $|x| \leq M$.
 \item $\varphi_A(r,x) < p/2$ for all $0 \leq r \leq R$ and $x$ with $|x| > M$.
\end{itemize}
Thus, $\tilde{J}_A(p)$ is the infimum of the continuous function $r$ over the 
non-empty compact set 
\begin{equation}
  \{(r,x) : \varphi_A(r,x) \geq p, \, 0 \leq r \leq R,\, |x| \leq M \} \,,
\end{equation}
which gives part~\ref{prop:supx2}.

Finally, if $p > \nu(A)$, then $\tilde{I}_A(p) > 0$ by part~\ref{prop:supx1}. At the same time, if $\tilde{J}_A(p)= 0$, then $\tilde{I}_A(p) < \infty$ by part~\ref{prop:supx2}. This takes care of part~\ref{prop:supx3}.
\end{proof}

Below is a standard result concerning the uniformity of the convergence to the Normal distribution under the CLT.
\begin{prop}
\label{prop:CLTUniformity}
Let $A \subseteq \bbR$ be a continuity set of $\nu(A)$, i.e. $\nu(\partial A) = 0$ and 
$R > 0$. Then,
\begin{equation}
  \lim_{n \to \infty} \sup_{\rho \in [R^{-1}, R]} \sup_{\xi \in \bbR}
    |\nu_n(\sqrt{n}(\rho A + \xi)) - \nu(\rho A + \xi)| = 0 \,.
\end{equation}
\end{prop}
\begin{proof}
By Theorem 2 in~\cite{BillFlem}, it is enough to check that 
\begin{equation}
\label{eqn:20.2}
  \lim_{\delta \to 0} \sup_{\xi, \rho} \nu\big( (\partial(\rho A + \xi))^\delta \big) = 0,
\end{equation}
where for a set $D \subset \bbR$, we set $D^{\delta} := \{x \in \bbR :\: \inf_{y \in D} |x-y| < \delta\}$ and the supremum is over $\rho$ and $\xi$ as in the statement in the proposition. Since $\nu$ is equivalent to $\lambda$, Lebesgue measure on $\bbR$, 
we may show~\eqref{eqn:20.2} with $\lambda$ in place of $\nu$. But,
\begin{equation}
  \lambda \big((\partial(\rho A + \xi))^\delta \big) = 
  \lambda \big(\rho(\partial A)^{\delta/\rho} + \xi \big) 
  \leq R \lambda \big( (\partial A)^{R \delta} \big) \,,
\end{equation}
The last term goes to $0$ as $\delta \to 0$, since $\lambda(\partial A) = 0$.
\end{proof}

We shall need the following uniform Chernoff-Cram\'{e}r-type upper bound. 
\begin{lem}
\label{lem:UnifCramer}
Let $\bbX$ be a family of random variables on $\bbR$ with zero mean such that for some
$\theta_0 > 0$
\begin{equation}
\label{eqn:1a}
  \sup_{X \in \bbX} \bbE e^{\theta_0 X} < \infty \quad \text{ and } \quad
  \sup_{X \in \bbX} \bbE (X^-)^2 < \infty \, .
\end{equation}
Then there exists $C > 0$ such that for any $\Delta > 0$ small enough,
any $m\geq 1$ and $X_1, \dots, X_m$ independent copies of random variables in $\bbX$
\begin{equation}
\label{eqn:2}
  \bbP \big(\tfrac{1}{m} \sum_{i=1}^m X_i \, > \, \Delta \big) \leq e^{-C \Delta^2 m}
\end{equation}
\end{lem}

\begin{proof}
Using the exponential Chebyshef's inequality we may bound the l.h.s. in \eqref{eqn:2}
for any $0 < \theta \leq \theta_1 < \theta_0$ by
\begin{equation}
\label{eqn:3}
  \exp \big\{-m \big(\Delta \theta - m^{-1} \sum_{i=1}^m L_{X_i}(\theta) \big) \big\} \, ,
\end{equation}
where we use $L_X(\theta) = \log \bbE e^{\theta X}$ for the log moment generating 
function of $X$. Since $L_X(\theta)$ is in $\rmC^{\infty}([0,\theta_0))$ due to \eqref{eqn:1a}, we may use Taylor expansion to write (note that the first two terms are $0$)
\begin{equation}
  L_{X}(\theta) = \tfrac12 L^{\prime\prime}_{X}(\wh{\theta}) \theta^2 \, ,
\end{equation}
for some $\wh{\theta} \in (0, \theta)$. Now if we denote by $M_X(\wh{\theta}) = \bbE e^{\wh{\theta} X}$ the moment generating function of $X$ then 
\begin{equation}
  L^{\prime\prime}_{X}(\wh{\theta}) = 
    \frac{M^{\prime\prime}_{X}(\wh{\theta})M_{X}(\wh{\theta}) - (M^\prime_{X}(\wh{\theta}))^2}{M^2_{X}(\wh{\theta})}
    < C M_{X}(\theta_0) + \bbE (X^-)^2\, .
\end{equation}
This follows since $M_{X}(\wh{\theta}) \geq 1$ via Jensen's inequality and since
\begin{equation}
  M^{\prime\prime}_{X}(\wh{\theta}) = \bbE X^2 e^{\wh{\theta} X} 
    \leq \bbE X^2 1_{X < 0} + C \bbE e^{\theta_0 X} 1_{X \geq 0} \,,
\end{equation}
for some $C>0$ independent of $X \in \bbX$. Therefore \eqref{eqn:1a} implies that there exists $K > 0$ for which
\begin{equation}
  \sup_{X \in \bbX} \sup_{\wh{\theta} \in (0, \theta_1)} L^{\prime\prime}_{X}(\wh{\theta}) < K \,
\end{equation}
and thus
\begin{equation}
  \Delta \theta - m^{-1} \sum_{i=1}^m L_{X_i}(\theta) \geq \Delta\theta - \tfrac12 K\theta^2 \,.
\end{equation}
Using this bound with $\theta = \Delta/K$ in \eqref{eqn:3} and assuming $\Delta$ is small enough, the result follows with $C=(2K)^{-1}$ in \eqref{eqn:2}.
\end{proof}

The last lemma can be used to prove the following.
\begin{lem}
\label{lem:ConcenNormalPop}
There exists $C, C^\prime >0$ such that for all $\Delta > 0$ sufficiently small, 
$A \subset \bbR$, $\zeta \in \cZ$ and $n \geq 1$,
\begin{equation}
\label{eqn:9}
  \bbP \left(\ol{Z}^{\zeta}_n(A) > \tfrac{1}{|\zeta|} 
    \sum_{x \in \zeta} \nu_n(A - x) + \Delta \right) \leq C e^{-C^\prime \Delta^2 |\zeta|} \,.
\end{equation}
The same holds if we replace $>$ with $<$ and $+\Delta$ with $-\Delta$.
\end{lem}

\begin{proof}
Starting with the first inequality and using 
\begin{equation}
  \ol{Z}_n^\zeta(A) = 
  \frac{\tfrac{1}{|\zeta|} \sum_{x \in \zeta} \wh{Z}_n^x(A)}
       {\tfrac{1}{|\zeta|} \sum_{x \in \zeta} |\wh{Z}_n^x| } \, ,
\end{equation}
the l.h.s. of \eqref{eqn:9} is bounded above by
\begin{equation}
\label{eqn:12}
    \bbP \big(\tfrac{1}{|\zeta|} \sum_{x \in \zeta} \wh{Z}_n^x(\bbR)  < 
      1-\tfrac{\Delta}{2} \big) \, + \, 
     \bbP \big(\tfrac{1}{|\zeta|} \sum_{x \in \zeta} \wh{Z}_n^x(A)  > 
       \tfrac{1}{|\zeta|} \sum_{x \in \zeta} \nu_n(A - x) + \tfrac{\Delta}{3} \big)
\end{equation}
as long as $\Delta$ is small enough. Now Theorem~4 in \cite{ATH1} gives a uniform bound
on the moment generating function $e^{\theta \wh{Z}_n(\bbR)}$ for all $n \geq 1$ and $\theta \in [0,\theta_0]$, for some $\theta_0 > 0$. This uniform bound can be extended to include
also the moment generating functions of (the stochastically smaller) $\wh{Z}_n^x(A)$ for all $A \subseteq \bbR$ and $x \in \bbR$ in the same range of $\theta$. The non-negativity of all these random variables imply that we may extend the bound also to all $\theta < 0$. Thus, it is not difficult to see that the family of random variables 
\begin{equation}
  \bbX = \{\pm(\wh{Z}_n(A) - \nu_n(A)) : \: n \geq 1,\, A \subseteq \bbR\}
\end{equation}
satisfies the conditions in Lemma~\ref{lem:UnifCramer}, whence \eqref{eqn:12}
is bounded above by $C e^{-C^\prime \Delta^2 |\zeta|}$ for some $C, C^\prime > 0$ as desired. 

Replacing $A$ with $A^\rmc$ , we obtain \eqref{eqn:9} 
with $<$, $-\Delta$ in place of $>$, $+\Delta$. 
\end{proof}

We shall need the following uniform lower bound on the probability of a typical deviation of $\ol{Z}_n$ from the Gaussian distribution. 
\begin{lem}
\label{lem:TypicalDev}
For all $A \in \cA$, $t > 0$ there exists $\epsilon = \epsilon(t,A) > 0$ such that
\begin{equation}
\label{eqn:144}
  \liminf_{n \to \infty} \bbP(\ol{Z}_n(\sqrt{n}A) > \nu(A) + t/\sqrt{n}) > \epsilon \, .
\end{equation}
Moreover, we may choose the $\epsilon$'s such that for fixed $A \in \cA$ and $t>0$, 
\begin{equation}
\label{eqn:145}
  \inf_{A^\prime} \epsilon(t, A^\prime) > 0 \,,
\end{equation}
and the above limit with $A^\prime$ in place of $A$ is uniform in $A^\prime$, where
$A^\prime = \rho A + \xi$ for $(\rho, \xi)$ in any compact subset of $(0,\infty) \times (-\infty, 
+\infty)$. The same result holds with $<$ in place of $>$ and $-t/\sqrt{n}$ in place 
of $+t/\sqrt{n}$.
\end{lem}
\begin{proof}
Consider $A^\prime = \rho A + \xi$ for some $(\rho, \xi) \in (0,\infty) \times (-\infty, +\infty)$.
We may write $\sqrt{n}(\ol{Z}_n(\sqrt{n}A^\prime) - \nu(A^\prime))$ as (recall the definition
of $|\wh{Z}|$ in \eqref{e:1003}),
\begin{equation}  
  \frac{\sqrt{n}(\wh{Z}_n(\sqrt{n}A^\prime)-|\wh{Z}|\nu(A^\prime))}{|\wh{Z}_n|} - \frac{\nu(A^\prime)}{|\wh{Z}_n|}
    \sqrt{n}(|\wh{Z}_n|-|\wh{Z}|)
\end{equation}
Now Theorem 4.2 in~\cite{revesz} states that
\begin{equation}
  \bbE (|\wh{Z}_n| - |\wh{Z}|)^2 = O(\beta^{-n})\, ,
\end{equation}	
from which it follows by Borel-Cantelli that $\sqrt{n}(|\wh{Z}_n| - |\wh{Z}|) \to 0$, a.s.
At the same time Corollary 2.3 in \cite{chen2001exact} (notice that the typo $O(1)$ instead of $o(1)$ there) implies that for some positive $C_0, C_1$,
\begin{equation}
\label{e:33b}
  \liminf_{n \to \infty} \sqrt{n}(\wh{Z}_n(\sqrt{n}A^\prime) - |\wh{Z}|\nu(A^\prime)) \geq
    -C_0 |\wh{Z}| - C_1 \wh{M} \quad \bbP\text{-a.s.,}
\end{equation}
where $\wh{M} = \lim_{n \to \infty} \wh{M}_n$ and $\wh{M}_n = \int x \dd \wh{Z}_n$. 
The sets considered in the corollary are of the form $(-\infty, y]$, but it is clear that by summation one can extend it to all sets in $\cA$. 
Furthermore, it is immediate from the statement of the corollary that the constants $C_0, C_1$ 
can be chosen independently of $A^\prime = \rho A + \xi$ as long as $(\rho, \xi)$ are chosen from a compact subset of $(0, \infty) \times (-\infty, +\infty)$. Less immediate, but still true, is that the proofs of the corollary and Theorem 2.2 on which it is based in fact give that the above limit is uniform in all such $A^\prime$. Combining all the above and writing $\ol{M}$ for $\wh{M}/|\wh{Z}|$ we have,
\begin{equation}
  \liminf_{n \to \infty} \inf_{A^\prime} \sqrt{n}(\ol{Z}_n(\sqrt{n}A^\prime) - \nu(A^\prime)) \geq -C_0 - C_1 \ol{M} \quad \bbP\text{-a.s.}
\end{equation}
and it remains to show that $\ol{M}$ is unbounded from below. 

To this end, note that
$\ol{M} = \lim_{n \to \infty} \ol{M}_n$ where $\ol{M}_n = \wh{M}_n/|\wh{Z}_n|$,
and that for any integers $r < n$, symmetry of $\ol{M}_{n-r}$ around zero entails
\begin{equation}
  \bbP(\ol{M}_n \leq -r) \geq \bbP(\ol{Z}_r = \delta_{-r}) \tfrac{1}{2} > C
\end{equation}
where $C=C(r) > 0$ does not depend on $n$. Therefore $\bbP(\ol{M} \leq -r) > 0$, and 
since $r$ is arbitrary, $\ol{M}$ is indeed unbounded. This shows \eqref{eqn:144} and \eqref{eqn:145}.

Finally, applying the above results to $A^\rmc$ in place of $A$, we obtain the same lower bound for 
the probability of a deviation to the opposite side.
\end{proof}

\subsection{Proof of Theorem \ref{thm:GeneralCase}}
\label{sub:ProofOfMainTheorem}
Fix $A$ and $p$ as in the conditions of the theorem. There are two cases to consider, according to whether $I_A(p)$ is finite or not.

\subsubsection{The Case $I_A(p) < \infty$.}
Let $x$ be such that $\nu(A-x) \ge p$ and $|x|= \tilde{I}_A(p) > 0$, as guaranteed by Proposition~\ref{prop:supx}.

\medskip
\noindent {\em Lower bound}. 
Set 
\begin{equation}
  w = \lfloor |x| \sqrt{n} \rfloor \sgn(x) \quad ; \qquad
  m = n-|w| \quad ; \qquad 
  \zeta = b^{|w|} \delta_{w}
\end{equation}
and write
\begin{equation}
\label{e:27}
 \bbP(\ol{Z}_n(\sqrt{n}A) \geq p) \geq 
    \bbP(Z_{|w|} = \zeta) \  \bbP(\ol{Z}^\zeta_{m}(\sqrt n A) \geq p )
\end{equation}
The first factor can be lower bounded by $\exp\{-C b^{|w|}\}$ as the event $\{Z_{|w|} = \zeta\}$
is equivalent to having all particles in the first $|w|$ generations give birth to $b$ children, all of whom take either a $+1$ step or a $-1$ step, depending on the sign of $x$. This requires that at most $C^\prime b^{|w|}$ independent particles make certain branching/walking choices, all of which have a uniformly positive probability.

The second factor in \eqref{e:27} can be bounded below by
\begin{equation}
  \label{e:28}
  \big(\bbP(\ol{Z}_m(\sqrt{n}A-w) \geq p)\big)^{|\zeta|} \, .
\end{equation}
The probability in the above expression is further bounded below by
\begin{equation}
  \bbP \left( \ol{Z}_{m}
    \left( \sqrt{m} \left( \sqrt{\frac{n}{m}}(A-x) + 
	\frac{x\sqrt{n}-w}{\sqrt{m}}
    \right) \right) \geq \nu \big( A-x \big) \right)\, .
\end{equation}
which, for $\rho = \sqrt{\frac{n}{m}}$ and $\xi = \frac{x\sqrt{n}-w}{\sqrt{m}}$, is
is equal to 
\begin{equation}
\label{e:30}
  \bbP \big(\ol{Z}_{m} \big(\sqrt{m}(\rho(A-x) + \xi) \big)
      \geq \nu(\rho(A-x) + \xi) + \nu(A-x) - \nu(\rho(A-x) + \xi)\big) \,.
\end{equation}
Now since $\rho = 1 + O(1/\sqrt{n})$, $\xi = O(1/\sqrt{n})$, part~\ref{prop:supx0} of 
Proposition~\ref{prop:supx} implies that $\nu(A-x) - \nu(\rho(A-x) + \xi) = O(1/\sqrt{n})$, whence
we may find $t > 0$ large enough such that \eqref{e:30} is bounded below by
\begin{equation}
  \bbP \big(\ol{Z}_{m} \big(\sqrt{m}(\rho(A-x) + \xi) \big)
      \geq \nu(\rho(A-x) + \xi) + t/\sqrt{m} \big) 
\end{equation}
This is bounded away from $0$ uniformly in $n$ via Lemma~\ref{lem:TypicalDev}. 

Plugging this back into \eqref{e:28}, recalling that $|\zeta| = b^{|w|}$,
the second factor in \eqref{e:27} is bounded below by
$\exp\{-C^\prime b^{|w|}\}$. Combining the bounds on both factors in \eqref{e:27}
we arrive at
\begin{equation}
 \bbP(\ol{Z}_n(\sqrt{n}A) \geq p) \geq \exp\{-C b^{|x|\sqrt{n}}\} = 
    \exp \big\{-e^{(\log b)\tilde{I}_A(p)\sqrt{n} + C^\prime} \big\}
\end{equation}
as desired.

\medskip
\noindent {\em Upper bound}. Let $\epsilon > 0$ be arbitrarily small and set
\begin{equation}
  |w_\epsilon| = \lfloor (|x|-\epsilon)\sqrt{n}\rfloor \quad ; \qquad m_\epsilon = n-|w_\epsilon| \, . 
\end{equation}
Conditioning on the particles measure $\zeta$ at generation $|w_\epsilon|$, we have
\begin{equation}
\label{e:33}
  \bbP(\ol{Z}_{n}(\sqrt{n}A) \geq p) =
    \sum_{\zeta} \bbP (\ol{Z}^\zeta_{m_\epsilon}(\sqrt{n}A) \geq p) \bbP(Z_{|w_\epsilon|} = \zeta) \, .
\end{equation}
Any such $\zeta$ must satisfy $\supp(\zeta) \subseteq [-|w_\epsilon|, +|w_\epsilon|]$.
Therefore there exists $\delta > 0$, such that for all such $\zeta$ and $z \in \zeta$,
\begin{equation}
  \nu(A-z/\sqrt{n}) \leq \max_{z :\: |z| \leq |x| - \epsilon} \nu(A-z) = p - \delta \, .
\end{equation}
This follows from the choice of $x$ and Proposition~\ref{prop:supx}. 

Using this proposition and also Proposition~\ref{prop:CLTUniformity}, we further obtain for $n$ large,
\begin{equation}
  \frac{1}{|\zeta|} \sum_{z \in \zeta} \nu_{m_\epsilon} (\sqrt{n}A - z) \leq
  \frac{1}{|\zeta|} \sum_{z \in \zeta} \nu \big( \sqrt{\tfrac{n}{m_\epsilon}} A - \tfrac{z}{\sqrt{m_\epsilon}}\big) + \frac{\delta}{2} < p - \frac{\delta}{3} \, .
\end{equation}
Then Lemma~\ref{lem:ConcenNormalPop} implies that $\bbP (\ol{Z}^\zeta_{m_\epsilon}(\sqrt{n}A) \geq p)$ is bounded above by
\begin{equation}
    \bbP \left(\ol{Z}^\zeta_{m_\epsilon}(\sqrt{n}A) \geq
      \frac{1}{|\zeta|} \sum_{z \in \zeta} \nu_{m_\epsilon} (\sqrt{n}A - z) + \frac{\delta}{3} \right) \leq C e^{-C^\prime|\zeta|} \,.
\end{equation}
As $|\zeta| \geq b^{|w_\epsilon|}$ we have from \eqref{e:33} for $n$ large enough,
\begin{equation}
  \bbP(\ol{Z}_{n}(\sqrt{n}A) \geq p) \leq \exp \big\{-e^{(\log b)(\tilde{I}_A(p) - \epsilon)\sqrt{n} - C}
    \big\} \,,
\end{equation}
and this concludes the upper bound as $\epsilon$ was arbitrary.

\subsubsection{The Case $I_A(p) = \infty$.}
The proof in this case is technically similar to the proof in the previous case, although the 
``optimal'' strategy for achieving the desired deviation is different. We start by setting
$r =\tilde{J}_A(p) \in (0,1)$ and choosing $x \in \bbR$ such that 
\begin{equation}
\label{e:38}
  \nu( (A-x)/\sqrt{1-r}) \geq p
\end{equation} 
This is guaranteed by Proposition~\ref{prop:supx}. 

\smallskip
\noindent {\em Lower bound}. Set 
\begin{equation}
q = 2\lfloor rn/2 \rfloor \quad ; \quad w = \lfloor |x| \sqrt{n} \rfloor \sgn(x) \quad ; \quad
s = q + |w| \quad ; \quad \zeta = b^{s} \delta_{w} 
\end{equation}
and write
\begin{equation}
\label{e:38.5}
  \bbP(\ol{Z}_n(\sqrt{n}A) \geq p) \geq \bbP(Z_{s} = \zeta) \bbP(\ol{Z}^\zeta_{n-s}(\sqrt{n}A) \geq p)\,.
\end{equation}
The first factor on the r.h.s. is at least $\exp\{-Cb^{s}\}$ since the event there can be achieved by having all particles give birth to $b$ children in the first $s$ generations, make only $+1$ or $-1$ steps in the first $|w|$ generations (depending on the sign of $x$), and then alternate between $+1$ and $-1$ steps in the succeeding $q$ generations. This requires that $C^\prime b^{s}$ independent particles make certain branching/walking choices, all of which have a uniformly positive probability.

The second factor is bounded below by 
\begin{equation}
\label{e:39}
  \big(\bbP(\ol{Z}_{n-s}(\sqrt{n}A - w) \geq p)\big)^{|\zeta|}
\end{equation}
Setting
\begin{equation}
  m = n-s \quad, \quad \rho = \sqrt{\frac{n}{m}(1-r)} \quad, \quad 
    \xi = \frac{x\sqrt{n} - w}{\sqrt{m}} \,,
\end{equation}
and using \eqref{e:38}, we may bound below the probability in \eqref{e:39} by
\begin{equation}
\begin{split}
  \bbP & \left[ \ol{Z}_m \left(\sqrt{m} \left(\rho\frac{A-x}{\sqrt{1-r}} + \xi\right)\right) 
    \right. \\
`	    & \qquad \left. \geq \nu \left(\left(\rho\frac{A-x}{\sqrt{1-r}} + \xi\right)\right)        	+ \nu \left( \frac{A-x}{\sqrt{1-r}} \right) - \nu \left(\left(\rho\frac{A-x}{\sqrt{1-r}}
	+ \xi\right)\right) \right]
\end{split}
\end{equation}
Now $\rho = 1+O(1/\sqrt{n})$ and $\xi = O(1/\sqrt{n})$ hence by Proposition~\ref{prop:supx}
part~\ref{prop:supx0}, there exists $t > 0$ for which the last probability is bounded below by
\begin{equation}
  \bbP \left[ \ol{Z}_m \left(\sqrt{m} \left(\rho\frac{A-x}{\sqrt{1-r}} + \xi\right)\right) 
    \geq \nu \left(\left(\rho\frac{A-x}{\sqrt{1-r}} + \xi\right)\right) + \frac{t}{\sqrt{m}}\right] \,.
\end{equation}
This is uniformly (in $n$, large enough) positive by virtue of Lemma~\ref{lem:TypicalDev}.
Therefore the second factor in \eqref{e:38.5} is bounded below by $e^{-C|\zeta|} \geq \exp\{-Cb^{s}\}$.

Plugging the two bounds in \eqref{e:38.5} we obtain
\begin{equation}
  \bbP(\ol{Z}_n(\sqrt{n}A) \geq p) \geq \exp \{-e^{(\log b)s + C}\}
    \geq \exp \big\{-e^{(\log b)\tilde{J}_A(p)n(1+o(1))} \big\} \, ,
\end{equation}
as desired.

\medskip
\noindent {\em Upper bound}. As in the previous case, let $\epsilon > 0$ be small enough and set 
\begin{equation}
  q_\epsilon = \lfloor (r-\epsilon)n\rfloor \quad ; \qquad m_\epsilon = n-q_\epsilon \, .
\end{equation}

This time we condition on the particles measure $\zeta$ at generation $q_\epsilon$:
\begin{equation}
\label{e:33a}
  \bbP(\ol{Z}_{n}(\sqrt{n}A) \geq p) =
    \sum_{\zeta} \bbP (\ol{Z}^\zeta_{m_\epsilon}(\sqrt{n}A) \geq p) \bbP(Z_{q_\epsilon} = \zeta) \, .
\end{equation}
Now, from the definition of $r$ it follows that there exists $\delta > 0$ 
such that for all $\epsilon^\prime \in [\epsilon, 2\epsilon]$ and $z \in \bbR$,
\begin{equation}
  \nu \left( \frac{A - z}{\sqrt{1-r+\epsilon^\prime}} \right) \leq p - \delta \,.
\end{equation}
Therefore, for any measure $\zeta$ and $n$ large enough by Propositions~\ref{prop:supx} and \ref{prop:CLTUniformity}
\begin{equation}
  \frac{1}{|\zeta|} \sum_{z \in \zeta} \nu_{m_\epsilon} (\sqrt{n}A - z) 
  \leq \frac{1}{|\zeta|} \sum_{z \in \zeta} \nu \big( \sqrt{\tfrac{n}{m_\epsilon}} A - 	\tfrac{z}{\sqrt{m_\epsilon}}\big) + \frac{\delta}{2}
  \leq p - \frac{\delta}{2} \, .
\end{equation}
Using Lemma~\ref{lem:ConcenNormalPop} we have that $\bbP(\ol{Z}^\zeta_{m_\epsilon}(\sqrt{n}A) \geq p)$ is bounded above by
\begin{equation}
\bbP \left(\bbP(\ol{Z}^\zeta_{m_\epsilon}(\sqrt{n}A) \geq
  \frac{1}{|\zeta|} \sum_{z \in \zeta} \nu_{m_\epsilon} (\sqrt{n}A - z) + \frac{\delta}{2}
  \right) \leq C e^{-C^\prime |\zeta|} \,.
\end{equation}
But if $\zeta$ is a possible particle
measure at generation $q_\epsilon$, then $|\zeta| \geq b^{q_\epsilon}$. Hence from
\eqref{e:33a} we obtain for $n$ large enough,
\begin{equation}
  \bbP(\ol{Z}_{n}(\sqrt{n}A) \geq p) \leq e^{-C b^{q_\epsilon}} \leq 
    \exp \big\{-e^{(\log b)(\tilde{J}_A(p)-\epsilon)n-C^\prime} \big\} \,,
\end{equation}
and since $\epsilon$ is arbitrary the upper bound follows.
\qed

\subsection{Proof of Proposition~\ref{prop:Interpolation}}
\label{sub:countersection}
Let $\alpha \in (1/2, 1)$ and $p \in (0,1)$ be given and choose $a>0$ such that 
$\nu(A_0) = p$ where $A_0 = [-a,+a]$. Fix some small $\delta > 0$ and for any integer $k \geq 1$ set:
\begin{equation}
  x_k = k^{1+\delta} \,,\quad 
  r_k = \sqrt{1 - k^{-\frac{(1-\alpha)(1+\delta)}{\alpha-1/2}}} \,,\quad
  A_k = x_k + r_k \cdot A_0 \,.
\end{equation}
Finally, for some $k_0 > 0$ to be chosen later, set
\begin{equation}
  A = \bigcup_{k={k_0}}^\infty A_k \,.
\end{equation}
We shall now argue that \eqref{eqn:111} is satisfied with the above $A$, $\alpha$ and $p$.

\medskip
{\em Lower bound.}
For any $n$ large enough, set $k=\lceil n^{(\alpha-1/2)/(1+\delta)} \rceil$,
$w = \lfloor x_k \sqrt{n} \rfloor$, $m=n-w$, $\zeta = b^w \delta_w$ and write
\begin{equation}
\label{ce:1}
 \bbP(\ol{Z}_n(\sqrt{n}A) \geq p) \geq
    \bbP(Z_{w} = \zeta) \  \bbP(\ol{Z}^\zeta_{m}(\sqrt n A) \geq p) \, .
\end{equation}
The first factor on the r.h.s. is at least $\exp\{-Cb^w\} \geq \exp \{-b^{n^{\alpha}(1+o(1)}\}$, as the event can be achieved by all particles multiplying at rate $b$ and having their descendants take a $+1$ step for $w$ generations. Therefore, as in the proof of the lower bound in the $I(A) < \infty$ case, it is enough to show that $\bbP(\ol{Z}_m(\sqrt{n}A-w) \geq p)$ is bounded away from $0$ independently of $n$. This, in turn, follows from Lemma~\ref{lem:TypicalDev} since
$\nu \big((\sqrt{n}A-w)/\sqrt{m} \big)$ is bounded below by
\begin{eqnarray}
  \nu \big( \sqrt{\tfrac{n}{m}} (A - x_k) \big) - O(n^{-1/2})  
    & \geq & \nu \big( \big(1-n^{-(1-\alpha)} \big)^{-1/2} (A_k - x_k) \big) - O(n^{-1/2}) \\
    & \geq & \nu \big( \big(1-n^{-(1-\alpha)} \big)^{-1/2} r_k \cdot A_0 \big) - O(n^{-1/2}) \\
    & \geq & \nu(A_0) - O(n^{-1/2}) \\
    & =    & p - O(n^{-1/2}) \,.
\end{eqnarray}

\medskip
{\em Upper bound}. Let $\epsilon > 0$ be arbitrarily small and set
$w_\epsilon = \lfloor (1-\epsilon) n^{\alpha} \rfloor$ and $m_\epsilon = n-w_\epsilon$.
By conditioning on the particles measure in generation $w_\epsilon$, it is clear that 
\begin{equation}
 \bbP(\ol{Z}_{n}(\sqrt{n}A) \geq p) \leq \max_{\zeta}\,
    \bbP (\ol{Z}^\zeta_{m_\epsilon}(\sqrt{n}A) \geq p)
\end{equation}
where the maximum is taken over all feasible particles measures $\zeta$ for generation $w_\epsilon$. For such $\zeta$, we may write 
\begin{align}
  \frac{1}{|\zeta|} \sum_{z \in \zeta} \nu_{m_\epsilon} (\sqrt{n}A - z) 
  &\leq \max_{z \in \zeta} \nu_{m_\epsilon} (\sqrt{n}A - z) \\
  &\leq \max_{z \in \zeta} \nu \big( \sqrt{\tfrac{n}{m_\epsilon}} 
  A - \tfrac{z}{\sqrt{m_\epsilon}}\big) + O(n^{-1/2}) \\
\label{eqn:113}
  &\leq  \max_{|y| \leq (1-\eps)n^{\alpha - 1/2}} 
      \nu \big( \sqrt{\tfrac{n}{m_\epsilon}} (A - y) \big)  + O(n^{-1/2})  \,,
\end{align}
where for the second inequality, we have used 
\begin{equation}
  \limsup_{m \to \infty} \sup_{\rho \in [1/2, 2]} \sup_{\xi \in \bbR}
    \, m^{1/2} |\nu_m(\sqrt{m}(\rho A + \xi)) - \nu(\rho A + \xi)| < \infty \, ,
\end{equation}
which holds for the set $A$ in light of (2.5) of \cite{bhattacharya1975errors}. 

Consider now some $y$ in the range of the maximum in \eqref{eqn:113} and find the index $k$ of the closest point to $y$ among $(x_k)_{k \geq k_0}$. We can then write 
\begin{equation}
\label{eqn:88}
  \sqrt{\tfrac{n}{m_\epsilon}} (A - y) = 
    \sqrt{\tfrac{n}{m_\epsilon}} (A_k - y)\cup
    \sqrt{\tfrac{n}{m_\epsilon}} (A \setminus A_k - y) 
\end{equation}
and bound the Gaussian measure of each set separately. 

The measure of the first set is upper bounded by (using Proposition~\ref{prop:supx})
\begin{eqnarray}
\nu \big( \sqrt{\tfrac{n}{m_\epsilon}} (A_k - x_k) \big)  
  & = & \nu \big( \sqrt{\tfrac{n}{m_\epsilon}} r_k \cdot A_0 \big) \\
  & \leq & \nu(A_0) - C \big(1 - \sqrt{\tfrac{n}{m_\epsilon}} r_k \big) \\
\label{eqn:91}
  & \leq & p - C \left(\big( 1-r_k \big) - \big(\sqrt{\tfrac{n}{m_\epsilon}} - 1 \big) \right) \,. 
\end{eqnarray}
For the second set in \eqref{eqn:88}, note that from the definition of $A$ it follows that
\begin{equation}
  A \setminus A_k - y \subseteq (-C k^\delta, +C k^\delta)^{\rmc} \,,
\end{equation}
for large enough $k$. Then, using a standard bound on the tails of $\nu$, we obtain
\begin{equation}
\label{eqn:92}
  \nu \big(\sqrt{\tfrac{n}{m_\epsilon}} (A \setminus A_k - y)\big)
    \leq C^\prime e^{-Ck^{2\delta}} \,.
\end{equation}
Combining the two bounds, we have
\begin{equation}
\label{eqn:93}
  \nu \big( \sqrt{\tfrac{n}{m_\epsilon}} (A - y) \big)  \leq 
      p - C^{\prime\prime} \left(\big( 1-r_k \big) - C^\prime e^{-C k^{2\delta}} - 
	\big(\sqrt{\tfrac{n}{m_\epsilon}} - 1 \big) \right)
\end{equation}
Now if $k_0$ is chosen large enough, the r.h.s. above is maximized when $k$ is the largest possible. At the same time, the choices of $k$ and $y$ imply
\begin{equation}
  (k-1)^{1+\delta}  < y \leq (1-\epsilon)n^{\alpha - 1/2}
\end{equation}
which gives an upper bound on $k$. Using this in \eqref{eqn:93} we infer that the r.h.s. of \eqref{eqn:113} is bounded above by
\begin{equation}
  p - C \left(  n^{-(1-\alpha)}/2 - (1-\epsilon) n^{-(1-\alpha)}/2 \right) (1-o(1))
  \leq p - C^\prime \epsilon n^{-(1-\alpha)} \, .
\end{equation}

We may now use Lemma~\ref{lem:ConcenNormalPop} and the fact that $|\zeta| \geq b^{w_\epsilon}$
to conclude that 
\begin{equation}
 \bbP (\ol{Z}^\zeta_{m_\epsilon}(\sqrt{n}A) \geq p) \leq  
    C \exp (-C' \epsilon^2 n^{-2(1-\alpha)} b^{(1-\eps)n^{\alpha}}) \,.
\end{equation}
This finishes the proof as $\eps$ was arbitrary.
\qed

\section*{Acknowledgments}
The authors would like to thank Dima Ioffe and Anna Levit for the discussions which gave rise to this problem, Scott Sheffield for fruitful conversations in the early stages of this project and Eve Styles for her interest in the problem.
\bibliographystyle{abbrv}
\bibliography{BRW}

\end{document}